\documentclass[9pt,shortpaper,twoside,web]{ieeecolor}

\usepackage{mathdots}
\usepackage{xcolor}
\let\labelindent\relax
\usepackage{enumitem}
\usepackage{generic}
\usepackage{comment}
\usepackage{mathtools,amsmath}
\mathtoolsset{showonlyrefs}
\usepackage{tabularx} 
\usepackage{diagbox}
\usepackage{graphicx}
\usepackage{mathrsfs}
\usepackage[noadjust]{cite}
\usepackage{amssymb,amsfonts}
\usepackage{algorithmic}
\usepackage{siunitx}
\usepackage{tikz}
\usetikzlibrary{calc}
\usepackage{pgfplots,tikz}
\usepackage{pgfplotstable}
\usepackage[american]{circuitikz}
\usetikzlibrary{calc,patterns,
                 decorations.pathmorphing,
                 decorations.markings}
\usetikzlibrary{shapes,arrows,patterns,positioning}
\tikzstyle{block} = [draw, fill=gray!20, rectangle, 
    minimum height=2em, minimum width=4em]
\tikzstyle{sum} = [draw, fill=gray!20, circle, node distance=1.5cm]
\tikzstyle{input} = [coordinate]
\tikzstyle{output} = [coordinate]
\tikzstyle{pinstyle} = [pin edge={to-,thin,black}]
\usepgfplotslibrary{units}
\usetikzlibrary{shapes,arrows,patterns,positioning,spy,backgrounds}
\tikzstyle{block} = [draw, fill=gray!20, rectangle, 
    minimum height=2em, minimum width=4em]
\tikzstyle{sum} = [draw, fill=gray!20, circle, node distance=1.5cm]
\tikzstyle{input} = [coordinate]
\tikzstyle{output} = [coordinate]
\tikzstyle{pinstyle} = [pin edge={to-,thin,black}]

\DeclareMathOperator{\rank}{rank}
\DeclareMathOperator{\tr}{tr}

\DeclareMathOperator{\im}{im}

\DeclareMathOperator{\spec}{spec}

\newcommand{\set}[2]{\left\{#1 \mid #2\right\}}

\usepackage{float}
\usepackage{accents}
\usepackage{multicol}
\usepackage{multirow}
\usepackage[ruled,vlined]{algorithm2e}

\usepackage{array,multirow}
\usepackage{hhline}
\usepackage{arydshln}
\usepackage{booktabs}

\newtheorem{theorem}{Theorem}

\newtheorem{lemma}[theorem]{Lemma}
\newtheorem{example}[theorem]{Example}

\newtheorem{proposition}[theorem]{Proposition}
\newtheorem{problem}{Problem}

\newtheorem{remark}[theorem]{Remark}

\newtheorem{definition}[theorem]{Definition}

\usepackage{textcomp}
\def\BibTeX{{\rm B\kern-.05em{\sc i\kern-.025em b}\kern-.08em
    T\kern-.1667em\lower.7ex\hbox{E}\kern-.125emX}}
\begin{document}
\title{Data-Driven Stabilization Using Prior Knowledge on Stabilizability and Controllability}
\author{Amir Shakouri, 
Henk J. van Waarde, Tren M.J.T. Baltussen, W.P.M.H. (Maurice) Heemels
\thanks{The work of Henk van Waarde was supported by
the Dutch Research Council under the NWO Talent Programme Veni
Agreement (VI.Veni.22.335). The work of Maurice Heemels was supported by the European Research Council (ERC) under the Advanced ERC grant agreement PROACTHIS, no. 101055384.}
\thanks{Amir Shakouri and Henk van Waarde are with the Bernoulli Institute for Mathematics, Computer Science and Artificial Intelligence, University of Groningen, The Netherlands (e-mail: a.shakouri@rug.nl; h.j.van.waarde@rug.nl). }
\thanks{Tren Baltussen and Maurice Heemels are with  the Control Systems Technology section, Department of Mechanical Engineering, Eindhoven
University of Technology, The Netherlands
(e-mail: t.m.j.t.baltussen@tue.nl; 
m.heemels@tue.nl).}
}

\maketitle

\thispagestyle{empty}
\pagestyle{empty}

\begin{abstract}
In this work, we study data-driven stabilization of linear time-invariant systems using prior knowledge of system-theoretic properties, specifically stabilizability and controllability. To formalize this, we extend the concept of data informativity by requiring the existence of a controller that stabilizes all systems consistent with the data \emph{and} the prior knowledge. We show that if the system is controllable, then incorporating this as prior knowledge does not relax the conditions required for data-driven stabilization. Remarkably, however, we show that if the system is stabilizable, then using this as prior knowledge leads to necessary and sufficient conditions that are \emph{weaker} than those for data-driven stabilization without prior knowledge. In other words, data-driven stabilization is easier if one knows that the underlying system is stabilizable. We also provide new data-driven control design methods in terms of linear matrix inequalities that complement the conditions for informativity.
\end{abstract}


\section{Introduction}
\label{sec:introduction}

In recent years, it has been shown that stabilizing feedback laws can be directly obtained from measured data, as opposed to the classical approach of using a model of the system (see \cite[Ch. 1.2]{DBLSCT2025} for a historical account). This idea can be motivated by the argument that bypassing the modeling procedure may reduce the total amount of computations since it does not require system identification as an intermediate step, cf. \cite{dorfler2022bridging}. In addition, finding a stabilizing feedback directly from the data might be feasible even when the data do not contain sufficient information for accurate modeling of the system, see \cite[Ex. 19]{van2020data}. 

For linear time-invariant (LTI) systems, direct data-driven stabilization has been extensively studied in the literature. It was shown in \cite{de2019formulas} that a stabilizing state-feedback can be directly obtained from input-state data by solving a linear matrix inequality (LMI). The method provided in \cite{de2019formulas} requires the input data to be persistently exciting of a certain order, see \cite[p. 327]{Willems2005}. This condition implies that the system can be uniquely identified. Soon after, it was shown in \cite{van2020data} that persistency of excitation is not necessary. In fact, the necessary and sufficient conditions studied in \cite{van2020data} make it possible to obtain a data-driven stabilizing feedback with minimal requirements on the data. Interestingly, such conditions may hold even if unique identification is not feasible. Data-driven stabilization in the presence of process and measurement noise has been studied in \cite{van2020noisy,kaminaga2025data,li2026controller}. Apart from computing stabilizing feedback gains, it has been shown that trajectory simulation \cite{markovsky2008data} and construction of predictive controllers \cite{coulson2019data,berberich2020data} can be performed directly from time series data. In addition, data-driven predictive control using frequency-domain data has also been studied in \cite{MeiNou_NMPC24a}.

The majority of papers on data-driven control work in the setting where the parameters of the system are completely unknown, which is interesting in its own right. However, for most physical systems, this is a rather conservative modeling framework. In fact, we often have access to some \mbox{\emph{prior knowledge}} on the system parameters, for example, because they represent physical quantities such as mass, spring constant, or conductance that are between given upper and lower bounds. For such cases, using both prior knowledge and data can lead to design methods that are less conservative than using data alone. For instance, in case computing a stabilizing feedback gain solely from the data is not feasible, prior knowledge could be used in conjunction with the data to enable such feedback design. This motivates developing methods that synthesize feedback laws by leveraging both the collected data and prior knowledge. Existing works on direct data-driven control that incorporate prior knowledge are rather scarce. So far, only prior knowledge in the form of bounds on the system parameters \cite{berberich2022combining,qi2023data,niknejad2025online} and exact knowledge of some parameters \cite{huang2025data} have been studied in the literature. 
In particular, it was shown in \cite{berberich2022combining} that if the prior knowledge admits a linear fractional representation, one can combine such knowledge with the data to design a feedback law by solving LMIs. Compared to direct data-driven control, the use of prior knowledge in system identification has a richer history. For instance, subspace identification using the system's stability as prior knowledge has been studied in \cite{van2002identification,lacy2003subspace}. This has been extended to incorporating eigenvalue constraints in \cite{miller2013subspace}. Prior knowledge on other system-theoretic properties, such as positivity \cite{de2002identification} and passivity \cite{goethals2003identification,shali2024towards}, is also among the investigated topics. The reader can refer to \cite{inoue2019subspace,khosravi2023kernel,10039069} and the references therein for other types of prior knowledge that have been used in system identification.

This note studies data-driven stabilization of LTI systems using prior knowledge on stabilizability and controllability. The incorporation of such system-theoretic properties has not yet received attention in data-driven control, and it poses significant technical challenges. This is, among others, due to the fact that sets of stabilizable and controllable systems are not convex, in contrast to the system sets considered in previous works \cite{berberich2022combining,qi2023data,niknejad2025online}. Nevertheless, the inclusion of this new type of prior knowledge is highly relevant because, in many cases, it is known a priori that the system is either controllable or stabilizable. This information can, for instance, be deduced from the structure of the system matrices, which has been studied in detail in the literature on \emph{structural controllability and stabilizability analysis} (see, e.g., \cite{jia2020unifying} and the references therein). 

In this work, we extend the data informativity framework of \cite{van2020data} to include prior knowledge by requiring the existence of a controller that stabilizes all systems consistent with the data \emph{and} the prior knowledge. Our main results are twofold. First, we show that data-driven stabilization using controllability as prior knowledge is equivalent to data-driven stabilization without prior knowledge (Theorem~\ref{th:1}). Therefore, if it is known that the true system is controllable, then this knowledge does not help in relaxing the conditions needed for data-driven stabilization. Next, we show that stabilizability as prior knowledge leads to necessary and sufficient conditions that are weaker than those of data-driven stabilization without prior knowledge (Theorems~\ref{th:12} and~\ref{th:3}). In addition, we provide a tractable method (Proposition~\ref{prop:K=[K1 K2]}), by which a stabilizing feedback gain can be computed from data while incorporating prior knowledge of stabilizability.  

\subsubsection*{Notation} Let $\mathbb{Z}$, $\mathbb{Z}_+$, and $\mathbb{N}$ denote the sets of integers, non-negative integers, and positive integers, respectively. We refer to the real part of a complex number $\lambda\in\mathbb{C}$ by $\mathfrak{R}(\lambda)$. Let \mbox{$M \in \mathbb{R}^{n \times m}$}. The Moore-Penrose pseudoinverse of $M$ is denoted by $M^\dagger$. The spectral norm of $M$ is denoted by $\|M\|$. For a set $\mathcal{X}\subseteq\mathbb{R}^{m}$, we define $M\mathcal{X}\coloneqq\set{Mx}{x\in\mathcal{X}}$. We denote the set of eigenvalues of a square matrix $M$ by $\spec M$. We say $M$ is \emph{positive definite}, denoted by $M > 0$, if it is symmetric and every $\lambda\in\spec M$ is positive. We denote the reachable subspace of a pair \mbox{$(A,B)\in\mathbb{R}^{n\times n}\times \mathbb{R}^{n\times m}$} by \mbox{$\mathcal{R}(A,B)\coloneqq \im\begin{bmatrix}
B & AB & \cdots & A^{n-1}B
\end{bmatrix}$}.


\section{Problem Formulation}
\label{sec:II}

Let $n,m\in\mathbb{N}$. Consider the LTI system
\begin{equation}
\label{eq:1}
x(t+1)=A_\text{true}x(t)+B_\text{true}u(t),
\end{equation}
referred to as the \emph{true system}, where $t\in\mathbb{Z}_+$ denotes time, \mbox{$x(t)\in\mathbb{R}^n$} is the state, and $u(t)\in\mathbb{R}^m$ is the input. The system matrices 
\begin{equation}
(A_\text{true},B_\text{true}) \in \mathcal{M} \coloneqq \mathbb{R}^{n \times n} \times \mathbb{R}^{n \times m}
\end{equation}
are assumed to be unknown. However, we have access to input-state data of the form
\begin{equation}
\label{eq:data}
\mathcal{D}\coloneqq\left(\begin{bmatrix}
    u(0)  & \cdots & u(T-1)
\end{bmatrix}, \begin{bmatrix}
    x(0) & \cdots & x(T)
\end{bmatrix}\right)
\end{equation} 
collected from \eqref{eq:1} within the time horizon $T\in\mathbb{N}$. Given $\mathcal{D}$, we define the matrices \mbox{$X_{-}\coloneqq \begin{bmatrix}
x(0) & x(1) & \cdots & x(T-1)
\end{bmatrix}$} and $X_{+}\coloneqq \begin{bmatrix}
x(1) & x(2) & \cdots & x(T)
\end{bmatrix}$. We also define the matrix \mbox{$U_- = \begin{bmatrix} u(0) & u(1) & \cdots & u(T-1) \end{bmatrix}$}. 

\subsection{Recap of Data-driven stabilization}

Roughly speaking, data-driven stabilization aims at solving the following problem:

\begin{center}
\textit{Given the data $\mathcal{D}$, find a $K$ such that $A_\textup{true}+B_\textup{true}K$ is Schur.}
\end{center}
Based on the collected input-state data, the true system satisfies 
\begin{equation}
\label{eq:lineqtrue}
X_+=A_\textup{true}X_-+B_\textup{true}U_-.
\end{equation}
However, the true system may not be the only one that satisfies this identity. Therefore, a feedback gain that guarantees the stability of the true system must stabilize \emph{all} systems $(A,B)$ satisfying \mbox{$X_+=AX_-+BU_-$}. Therefore, we define the set of \emph{data-consistent systems} as
\begin{equation}
\label{eq:exp}
\Sigma_\mathcal{D}\coloneqq \set{(A,B)\in\mathcal{M}}{X_+=AX_-+BU_-},
\end{equation}
and we sharpen the data-driven stabilization problem as follows:
\begin{center}
\textit{Given the data $\mathcal{D}$, find a $K$ such that $A+BK$ is Schur for all $(A,B)\in\Sigma_\mathcal{D}$.}
\end{center}
The feasibility of this problem depends on the given data. To elaborate on this, we recap the following notion of data informativity. 

\begin{definition}[{\cite[Def. 12]{van2020data}}]
\label{def:1}
The data $\mathcal{D}$ are called \emph{informative for stabilization} if there exists a $K\in\mathbb{R}^{m\times n}$ such that $A+BK$ is Schur for all $(A,B)\in\Sigma_\mathcal{D}$. 
\end{definition}

The following result provides a necessary and sufficient LMI condition for the informativity of the data for stabilization.

\begin{proposition}[{\cite[Thm. 17]{van2020data}}]
\label{prop:1}
The data $\mathcal{D}$ are informative for stabilization \emph{if and only if} there exists $\Theta\in\mathbb{R}^{T\times n}$ such that
\begin{equation}
\label{eq:dd_lmi}
X_-\Theta=\Theta^\top X_-^\top\ \text{ and }\ \begin{bmatrix}
X_-\Theta & X_+\Theta \\
\Theta^\top X_+^\top & X_-\Theta
\end{bmatrix}>0.
\end{equation}
Moreover, $A+BK$ is Schur for all $(A,B)\in\Sigma_\mathcal{D}$ \emph{if and only if} $K=U_-\Theta(X_-\Theta)^{-1}$ for some $\Theta$ satisfying \eqref{eq:dd_lmi}. 
\end{proposition}

Based on Proposition~\ref{prop:1}, a necessary condition for the informativity of the data for stabilization is that $\rank X_-=n$. This condition requires the number of data samples to satisfy $T\geq n$.

\subsection{Data-driven stabilization using prior knowledge}

Let $\Sigma_\text{pk}\subseteq\mathcal{M}$ be a set capturing our prior knowledge of the true system, i.e.,
\begin{equation}
(A_\text{true},B_\text{true})\in\Sigma_\text{pk}.
\end{equation}
Using this prior knowledge, we extend the data-driven stabilization problem to the following.

\begin{center}
\textit{Given the data $\mathcal{D}$ and the set of prior knowledge $\Sigma_\textup{pk}$, find a $K$ such that $A+BK$ is Schur for all $(A,B)\in\Sigma_\mathcal{D}\cap\Sigma_\textup{pk}$.}
\end{center}

The feasibility of this problem depends on the given data \emph{and} the prior knowledge. To study this problem, we extend the notion of data informativity in Definition~\ref{def:1} to the following, which takes the prior knowledge into account. 

\begin{definition}
\label{def:2}
The data $\mathcal{D}$ are called $\Sigma_\textup{pk}$--\emph{informative for stabilization} if there exists a $K\in\mathbb{R}^{m\times n}$ such that $A+BK$ is Schur for all $(A,B)\in\Sigma_\mathcal{D}\cap\Sigma_\text{pk}$. 
\end{definition}

We note that $\mathcal{M}$--informativity for stabilization is equivalent to informativity for stabilization in the sense of Definition~\ref{def:1}.  In this work, we are interested in two important sets of prior knowledge that capture the \emph{stabilizability} and \emph{controllability} of the true system. We denote by $\Sigma_\text{cont}, \Sigma_\text{stab} \subset \mathcal{M}$ the subsets of controllable and stabilizable systems, respectively.

The following example demonstrates that, in case
\begin{equation}
\label{eq:pk=stab}
\Sigma_\text{pk}=\Sigma_\text{stab},
\end{equation}
the conditions for $\Sigma_\text{pk}$-informativity for stabilization are in general \emph{weaker} than those for informativity for stabilization.

\begin{example}
\label{ex:1}
Consider the input data $u(0)=1$, \mbox{$u(1)=2$}, and $u(2)=-1$, and the state data $x(0)=\begin{bmatrix}
1 & 0
\end{bmatrix}^\top$, \mbox{$x(1)=\begin{bmatrix}
2 & 0
\end{bmatrix}^\top$}, $x(2)=\begin{bmatrix}
4 & 0
\end{bmatrix}^\top$, and $x(3)=\begin{bmatrix}
3 & 0
\end{bmatrix}^\top$. The set of data-consistent systems reads
\begin{equation}
\Sigma_\mathcal{D}=\set{\left(\begin{bmatrix}
1 & \alpha \\
0 & \beta
\end{bmatrix},\begin{bmatrix}
1 \\ 0
\end{bmatrix}\right)}{\alpha,\beta\in\mathbb{R}}.
\end{equation}
It follows from Proposition~\ref{prop:1} that the data are not informative for stabilization since $X_-=\begin{bmatrix}
1 & 2 & 4 \\
0 & 0 & 0
\end{bmatrix}$ does not have full row rank. However, we have
\begin{equation}
\Sigma_\mathcal{D}\cap\Sigma_\text{stab}=\set{\left(\begin{bmatrix}
1 & \alpha \\
0 & \beta
\end{bmatrix},\begin{bmatrix}
1 \\ 0
\end{bmatrix}\right)}{\alpha\in\mathbb{R},|\beta|<1}.
\end{equation}
It is evident that $K=\begin{bmatrix}
-1 & 0
\end{bmatrix}$ is a stabilizing feedback gain for all the systems in \mbox{$\Sigma_\mathcal{D}\cap\Sigma_\text{stab}$}. Therefore, the data are \mbox{$\Sigma_\text{stab}$--informative} for stabilization. 
\end{example}

Example~\ref{ex:1} shows that data-driven stabilization using stabilizability as prior knowledge may be possible in case Proposition~\ref{prop:1} fails to provide a feedback gain for all data-consistent systems. This motivates the study of stabilizability as prior knowledge for data-driven stabilization. Another closely related prior knowledge that is studied in this paper is controllability. Formally, we thus consider the following problem. 
\begin{problem}
\label{prob:1}
Find necessary and sufficient conditions under which the data $\mathcal{D}$ are (i) $\Sigma_\text{cont}$--informative for stabilization; \mbox{(ii) $\Sigma_\text{stab}$--informative} for stabilization.
\end{problem}

\section{Controllability as Prior Knowledge}

The main result of this section is the following theorem presenting the solution for Problem~\ref{prob:1}(i). This theorem shows that controllability as prior knowledge is \emph{not} useful, i.e., the data-driven stabilization using this prior knowledge is equivalent to data-driven stabilization without any prior knowledge. 

\begin{theorem}
\label{th:1}
Suppose that $(A_\text{true},B_\text{true})\in\Sigma_\text{cont}$. Then, the following statements are equivalent:
\begin{enumerate}[label=(\alph*),ref=\ref{th:1}(\alph*)]
    \item The data $\mathcal{D}$ are $\Sigma_\text{cont}$--informative for stabilization.
    \item The data $\mathcal{D}$ are informative for stabilization.
\end{enumerate}
Moreover, if $K$ is such that $A+BK$ is Schur for all \mbox{$(A,B)\in\Sigma_\mathcal{D}\cap\Sigma_\text{cont}$}, then $A+BK$ is Schur for all $(A,B)\in\Sigma_\mathcal{D}$. 
\end{theorem}

The following example provides an intuition on the fact that controllability as prior knowledge does not affect the conditions required for data-driven stabilization. 

\begin{example}
\label{ex:2}
For the sake of illustration, consider a system with $n=1$ and $m=2$ as follows:
\begin{equation}
x(t+1)=ax(t)+\begin{bmatrix}
b_1 & b_2
\end{bmatrix}u(t).
\end{equation}
Starting from $x(0)=-1$, we apply $u(0)=\begin{bmatrix}
1 & -1
\end{bmatrix}^\top$ and we measure $x(1)=-1$. Given these data, the set $\Sigma_\mathcal{D}$ consists of all systems with parameters $a$, $b_1$, and $b_2$ lying on the plane shown in Fig. 1. The only uncontrollable system on this plane corresponds to $a=1$ and $b_1=b_2=0$, which is shown by the red dot. Now, Theorem~\ref{th:1} states that if a feedback gain stabilizes all systems on the plane excluding the one shown in red, then it also stabilizes the red point. 
\end{example}

\begin{figure}[h]
    \centering
    \includegraphics[width=1\linewidth]{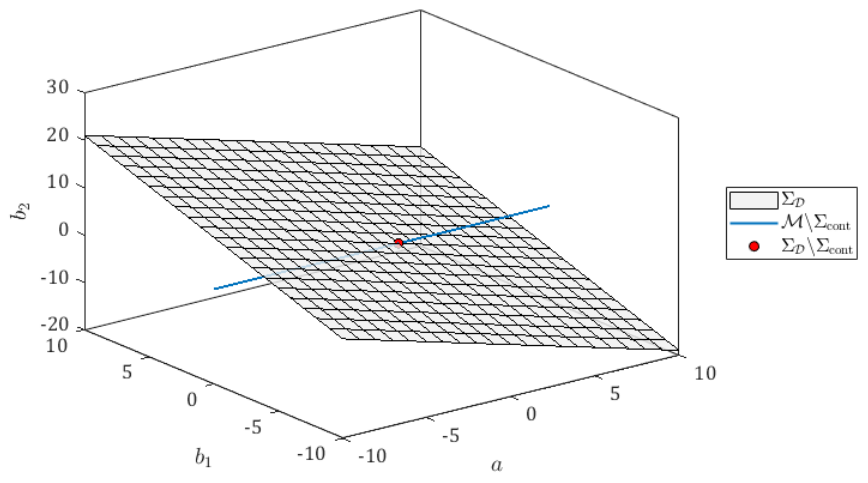}
    \caption{Set of data-consistent systems, set of uncontrollable systems, and their intersection for Example~\ref{ex:2}.}
    \label{fig:1}
\end{figure}

To prove Theorem~\ref{th:1}, we need three auxiliary results. The first result is the following lemma, showing that if a parameterized family of systems is controllable at a single point, then it is controllable at almost all points. 

\begin{lemma}
\label{lem:cont_direct}
Let $(M,N)\in\Sigma_{\text{cont}}$, $M_0\in\mathbb{R}^{n\times n}$, and $N_0\in\mathbb{R}^{n\times m}$. Then, the pair $(M+\alpha M_0,N+\alpha N_0)$ is controllable for all but at most $n^2$ values of $\alpha\in\mathbb{R}$.

\end{lemma}
\begin{proof}
Denote the Kalman controllability matrix of the pair $(M+\alpha M_0,N+\alpha N_0)$ by
\begin{equation}
\mathcal{C}(\alpha)\coloneqq\begin{bmatrix}
N+\alpha N_0 & \cdots & (M+\alpha M_0)^{n-1}(N+\alpha N_0)
\end{bmatrix}.
\end{equation}
Let $\bar{\mathcal{C}}(\alpha)\in\mathbb{R}^{n\times n}$ be a square submatrix of $\mathcal{C}(\alpha)$ such that $\bar{\mathcal{C}}(0)$ is nonsingular. Observe that the entries of $\bar{\mathcal{C}}(\alpha)$ are polynomials of $\alpha$ of degree at most $n$. It can be concluded from, e.g., Leibniz's formula for determinants, that $\det(\bar{\mathcal{C}}(\alpha))$ is a polynomial of $\alpha$ of degree at most $n^2$. This polynomial has at most $n^2$ roots counting multiplicity, which implies that $\rank \bar{\mathcal{C}}(\alpha)<n$ for at most $n^2$ distinct values \mbox{of $\alpha$}. Therefore, $\mathcal{C}(\alpha)$ is rank deficient for at most $n^2$ distinct values of $\alpha$.
\end{proof}

The second auxiliary result is the following lemma, providing a necessary and sufficient condition for two matrices to share the same eigenvalues.

\begin{lemma}[{\cite[2.4.P10]{horn2012matrix}}]
\label{lem:horn_2.4.P10}
Let $M,N\in\mathbb{R}^{n\times n}$. Then, $M$ and $N$ share the same eigenvalues with the same algebraic multiplicities \emph{if and only if} for every $k\in\{1,\ldots,n\}$ we have $\tr(M^k)=\tr(N^k)$. 
\end{lemma}

The third auxiliary result discusses the stability of matrix pencils, which is presented in the following lemma.

\begin{lemma}
\label{lem:A+delta*B}
Let $\varepsilon\in\mathbb{R}$, $\mathcal{F}\subset\mathbb{R}$ be a finite set, and $M,N\in\mathbb{R}^{n\times n}$ be such that $M+\delta N$ is Schur for all $\delta\in[\varepsilon,\infty)\backslash\mathcal{F}$. Then, $N$ is nilpotent and $M+\delta N$ is Schur for all $\delta\in\mathbb{R}$. 
\end{lemma}
\begin{proof}
Let $\bar{\delta}>0$ be sufficiently large such that $[\bar{\delta},\infty)\cap\mathcal{F}=\varnothing$. We note that the trace of a matrix is the sum of its eigenvalues. Since $M+\delta N$ is Schur for all $\delta\in[\bar{\delta},\infty)$, we have $|\tr ((M+\delta N)^k)|\leq n$ for all $\delta\in[\bar{\delta},\infty)$ and all $k\in\mathbb{N}$. Note that for every $k\in\mathbb{N}$, \mbox{$p_k(\delta)= \tr ((M+\delta N)^k)$} is a polynomial of degree at most $k$. The boundedness of $p_k(\delta)$ in the interval $[\bar{\delta},\infty)$ implies that $p_k(\delta)$ is constant, i.e., the coefficients of $\delta^i$ are zero for all $i\in\{1,\ldots,k\}$. In particular, the coefficient corresponding to $\delta^k$, i.e., $\tr(N^k)$, is equal to zero. Thus, $\tr(N^k)=0$ for all $k\in\mathbb{N}$. Based on Lemma~\ref{lem:horn_2.4.P10}, this implies that all eigenvalues of $N$ are equal to zero. Thus, $N$ is nilpotent. In addition, since $p_k(\delta)$ is constant for all $k\in\mathbb{N}$, we have $p_k(\delta)=p_k(0)$, thus, \mbox{$\tr(M^k)=\tr((M+\delta N)^k)$} for all $k\in\mathbb{N}$. Now, we use Lemma~\ref{lem:horn_2.4.P10} to conclude that $M$ and $M+\delta N$ share the same eigenvalues for all $\delta\in\mathbb{R}$. Therefore, $M+\delta N$ is Schur for all $\delta\in\mathbb{R}$. 
\end{proof}

The proof of Theorem~\ref{th:1} now follows from Lemmas~\ref{lem:cont_direct} and~\ref{lem:A+delta*B}. To facilitate the proof, we introduce the following notation:
\begin{equation}
\Sigma_\mathcal{D}^0\coloneqq\set{(A_0,B_0)\in\mathcal{M}}{A_0X_-+B_0U_-=0}.
\end{equation}
It is evident that the set $\Sigma_\mathcal{D}^0$ satisfies $\Sigma_\mathcal{D}^0+\Sigma_\mathcal{D}=\Sigma_\mathcal{D}$.

\textit{Proof of Theorem~\ref{th:1}:} 
(b)$\Rightarrow$(a): This implication is evident since if $K\in\mathbb{R}^{m\times n}$ is such that $A+BK$ is Schur for all $(A,B)\in\Sigma_\mathcal{D}$, then $A+BK$ is Schur for all $(A,B)\in\Sigma_\mathcal{D}\cap\Sigma_\text{cont}$.  

(a)$\Rightarrow$(b): This implication obviously holds if $\Sigma_\mathcal{D}\subseteq\Sigma_\text{cont}$. Now, assume that there exists $(\bar{A},\bar{B})\in\Sigma_\mathcal{D}$ that is not controllable. Let $K$ be such that $A+BK$ is Schur for all $(A,B)\in\Sigma_\mathcal{D}\cap\Sigma_\text{cont}$. It suffices to show that $\bar{A}+\bar{B}K$ is Schur. To that end, note that the pair $(\bar{A},\bar{B})$ satisfies $X_+=\bar{A}X_-+\bar{B}U_-$. Since the true system also satisfies $X_+=A_\textup{true}X_-+B_\textup{true}U_-$, we have
$\bar{A}=A_\text{true}+A_0$ and $\bar{B}=B_\text{true}+B_0$
for some \mbox{$(A_0,B_0)\in\Sigma_\mathcal{D}^0$}. We note that \mbox{$(A_\text{true}+\alpha A_0,B_\text{true}+\alpha B_0)\in\Sigma_\mathcal{D}$} for all $\alpha\in\mathbb{R}$. Based on Lemma~\ref{lem:cont_direct}, since $(A_\text{true},B_\text{true})\in\Sigma_\text{cont}$, we have \linebreak $(A_\text{true}+\alpha A_0,B_\text{true}+\alpha B_0)\in\Sigma_\mathcal{D}\cap\Sigma_\text{cont}$ for all but a finite number of $\alpha \in \mathbb{R}$. Thus, $A_\text{true}+B_\text{true}K+\alpha(A_0+B_0 K)$ is Schur for all but a finite number of $\alpha \in \mathbb{R}$. Based on Lemma~\ref{lem:A+delta*B}, this implies that $A_\text{true}+B_\text{true}K+\alpha(A_0+B_0 K)$ is Schur for all $\alpha\in\mathbb{R}$. We take $\alpha=1$ and conclude that $\bar{A}+\bar{B}K$ is Schur. \hfill \QED

\section{Stabilizability as Prior Knowledge}
\label{sec:IV}

In this section, we provide the solution for Problem~\ref{prob:1}(ii). First, we study necessary conditions for $\Sigma_\text{stab}$--informativity of the data for stabilization in Section~\ref{sec:IV-0}. Next, we provide necessary and sufficient conditions for cases $\rank X_-=n$ and $\rank X_-<n$ in Sections~\ref{sec:IV-A} and~\ref{sec:IV-B}, respectively. 

\subsection{Necessary conditions}
\label{sec:IV-0}

The following theorem presents four necessary conditions for $\Sigma_\text{stab}$--informativity for stabilization. 

\begin{theorem}
\label{th:nec_cond}
Suppose that $(A_\text{true},B_\text{true})\in\Sigma_\text{stab}$ and the data $\mathcal{D}$ are \mbox{$\Sigma_\text{stab}$--informative} for stabilization. Let $K$ be such that $A+BK$ is Schur for all \mbox{$(A,B)\in\Sigma_\mathcal{D}\cap\Sigma_\text{stab}$}. Then, the following statements hold:
\begin{enumerate}[label=(\alph*),ref=\ref{th:nec_cond}(\alph*)]
    \item\label{th:nec_cond(b)} $(A_0+B_0K)\mathcal{R}(A,B)=\{0\}$ for all $(A,B)\in\Sigma_\mathcal{D}\cap\Sigma_\text{stab}$ and all $(A_0,B_0)\in\Sigma_\mathcal{D}^0$. 
    \item\label{th:nec_cond(c)} If $\rank X_-<n$, then $\im \begin{bmatrix}
    X_- \\ U_-
    \end{bmatrix}=\im X_- \times \mathbb{R}^m$. 
    \item\label{th:nec_cond(d)} $\im X_+\subseteq \im X_-$.
    \item\label{th:nec_cond(e)} $\im X_-$ is $A$--invariant and contains $\im B$ for all $(A,B)\in\Sigma_\mathcal{D}$.
\end{enumerate}
\end{theorem}

The following remark makes a comparison between the necessary conditions in Theorem \ref{th:nec_cond} and the necessary conditions for data informativity without prior knowledge. 

\begin{remark}
It was shown in \cite[Lem. 15]{van2020data} that if the data $\mathcal{D}$ are informative for stabilization and $K$ is a stabilizing gain for all system in $\Sigma_\mathcal{D}$, then $A_0+B_0K=0$ for all \mbox{$(A_0,B_0)\in\Sigma_\mathcal{D}^0$}. Here, we see that if the data are \mbox{$\Sigma_\text{stab}$--informative} for stabilization, then this condition is relaxed to that of statement (a). If the true system is controllable, then statement (a) implies \mbox{$A_0+B_0K=0$} for all $(A_0,B_0) \in \Sigma_{\mathcal{D}}^0$. In other words, in this case, Theorem~\ref{th:nec_cond(b)} reduces to the necessary condition for informativity for stabilization in \cite[Lem. 15]{van2020data}. Moreover, we recall from Proposition~\ref{prop:1} that, without using prior knowledge, a necessary condition for the informativity of the data for stabilization is that $\rank X_-=n$. In that case, statements (c) and (d) obviously hold since $\im X_-=\mathbb{R}^n$. However, these statements, along with (b), are nontrivial in case the data are not informative for stabilization, but \mbox{$\Sigma_\text{stab}$--informative} for stabilization. 
\end{remark}

To prove Theorem \ref{th:nec_cond}, we need some intermediate results presented next. 

\begin{lemma}
\label{lem:imker_all}
Let $\mathcal{V} \subset \mathbb{R}^n$ be a proper subspace. Let $\hat{N}\in\mathbb{R}^{n\times m}$, $N_0\in\mathbb{R}^{r\times m}$, and $\varepsilon>0$. Define
\begin{equation}
\mathcal{N}\coloneqq\set{\hat{N}+YN_0}{Y\in\mathbb{R}^{n\times r},\|Y\|\leq \varepsilon}.
\end{equation}
Then, $\im N\subseteq\mathcal{V}$ for all $N\in\mathcal{N}$ \emph{if and only if} $\im \hat{N}\subseteq \mathcal{V}$ and $N_0=0$.  
\end{lemma}
\begin{proof}
The ``if'' part is obvious. To prove the ``only if'' part, we assume that $\im N\subseteq \mathcal{V}$ for all $N\in\mathcal{N}$, i.e.,
\begin{equation}
\label{eq:NimQ1}
\im (\hat{N} +YN_0)\subseteq \mathcal{V}
\end{equation}
for all $Y\in\mathbb{R}^{n\times r}$ with $\|Y\|\leq \varepsilon$. Taking $Y=0$ shows that \mbox{$\im \hat{N}\subseteq \mathcal{V}$}. This, together with \eqref{eq:NimQ1}, implies that
\begin{equation}
\label{eq:NimQ2}
Y\im N_0\subseteq \mathcal{V}
\end{equation}
for all $Y\in\mathbb{R}^{n\times r}$ with $\|Y\|\leq \varepsilon$. Let $\eta\in\mathbb{R}^r$. We take $Y=\xi\eta^\top$ with nonzero $\xi\in\mathbb{R}^n$ satisfying $\xi^\top\mathcal{V}=\{0\}$ and $\|\xi\eta^\top\|\leq \varepsilon$. We substitute this in \eqref{eq:NimQ2} and premultiply by $\xi^\top$ to obtain \mbox{$\|\xi\|^2\eta^\top \im N_0\subseteq \xi^\top\mathcal{V}=\{0\}$}. Since $\eta$ was arbitrary, we have $\eta^\top N_0=0$ for all $\eta\in\mathbb{R}^n$, which implies $N_0=0$. 
\end{proof}

\begin{lemma}
\label{lem:imX_-imX_+}
Let $\mathcal{V}\subseteq\mathbb{R}^n$ be a subspace of dimension $r\leq n$. Let $M\in\mathbb{R}^{n\times n}$ and $v\in\mathcal{V}$. If $M^kv\in\mathcal{V}$ for all $k\in[1,r]$, then $M^kv\in\mathcal{V}$ for all $k\in\mathbb{N}$. 
\end{lemma}
\begin{proof}
Suppose that $M^kv\in\mathcal{V}$ for all $k\in[1,r]$. We use induction to show that $M^kv\in\mathcal{V}$ for all $k\in\mathbb{N}$. Let $j\geq r$ be such that $M^{k} v\in \mathcal{V}$ for all $k\in[1,j]$. Since the dimension of $\mathcal{V}$ is equal to $r$, the matrix $\begin{bmatrix}
v & Mv & \cdots & M^{j} v
\end{bmatrix}$ is of rank at most $r$. Since the number of columns of this matrix is larger than $r$, there exists $i\in[1,j]$ such that $M^i v\in\im \begin{bmatrix}
v & Mv & \cdots & M^{i-1} v
\end{bmatrix}$. Multiply this from left by $M^{j-i+1}$ to obtain
\begin{equation}
M^{j+1} v\in\im \begin{bmatrix}
M^{j-i+1}v & M^{j-i+2}v & \cdots & M^j v
\end{bmatrix}\subseteq \mathcal{V}.
\end{equation}
Therefore, $M^{k} v\in \mathcal{V}$ for all $k\in[1,j+1]$, which completes the proof. 
\end{proof}

Lemmas \ref{lem:imker_all} and \ref{lem:imX_-imX_+} can now be used to prove Theorem \ref{th:nec_cond}. 

\textit{Proof of Theorem \ref{th:nec_cond}:}
(a) Suppose that $\Sigma_\mathcal{D}\cap\Sigma_\text{cont}$ is nonempty. Since the data are \mbox{$\Sigma_\text{stab}$--informative} for stabilization, they are also $\Sigma_\text{cont}$--informative for stabilization. Thus, we use Theorem~\ref{th:1} to conclude that the data are informative for stabilization and $A+BK$ is Schur for all $(A,B)\in\Sigma_\mathcal{D}$. It follows from \cite[Lem. 15]{van2020data} that we have $A_0+B_0K=0$ for all $(A_0,B_0)\in\Sigma_\mathcal{D}^0$. Therefore, statement~(a) holds. Now, suppose that $\Sigma_\mathcal{D}\cap\Sigma_\text{cont}$ is empty. Let $(A,B)\in\Sigma_\mathcal{D}$ be stabilizable but not controllable. Let $T\in\mathbb{R}^{n\times n}$ be nonsingular such that
\begin{equation}
\label{eq:Kalman_decomp_T}
TAT^{-1}=\begin{bmatrix}
A_{11} & A_{12} \\
0      & A_{22}
\end{bmatrix} \hspace{0.25 cm} \text{and} \hspace{0.25 cm} TB=\begin{bmatrix}
B_1 \\ 0
\end{bmatrix},
\end{equation}
where $(A_{11},B_1)\!\in\!\mathbb{R}^{n_1\times n_1}\times \mathbb{R}^{n_1\times m}$ is controllable, $A_{22}\!\in\!\mathbb{R}^{n_2\times n_2}$ is Schur, and $n_1+n_2=n$. Also, let $(A_0,B_0)\!\in\!\Sigma_\mathcal{D}^0$ and $Z\!\in\!\mathbb{R}^{n_1\times n}$. Define
\begin{equation}
\mathcal{A}(\alpha)\coloneqq A+\alpha T^{-1}\begin{bmatrix}
Z\\0
\end{bmatrix}A_0\text{ and } 
\mathcal{B}(\alpha)\coloneqq B+\alpha T^{-1}\begin{bmatrix}
Z\\0
\end{bmatrix}B_0.
\end{equation}
First, we claim that
\begin{equation}
\label{eq:claim1}
\left(\mathcal{A}(\alpha),\mathcal{B}(\alpha)\right)\in\Sigma_\mathcal{D}\cap\Sigma_\text{stab}
\end{equation}
for all but at most $n_1^2$ values of $\alpha$. To show this, observe that 
\begin{equation}
\mathcal{A}(\alpha) X_-+\mathcal{B}(\alpha)U_-=AX_-+BU_-=X_+.
\end{equation}
This implies that $(\mathcal{A}(\alpha),\mathcal{B}(\alpha))\in\Sigma_\mathcal{D}$ for all $\alpha\in\mathbb{R}$. Now, let the matrices $R_1\in\mathbb{R}^{n_1\times n_1}$ and $R_2\in\mathbb{R}^{n_1\times n_2}$ be defined by $\begin{bmatrix}
R_1 & R_2
\end{bmatrix}=ZA_0T^{-1}$. Observe that
    \begin{equation}
    \begin{split}
    T\mathcal{A}(\alpha)T^{-1}&=\begin{bmatrix}
    A_{11}+\alpha R_1 & A_{12}+\alpha R_2 \\
    0            & A_{22}
    \end{bmatrix}\text{ and} \\ T\mathcal{B}(\alpha)&=\begin{bmatrix}
    B_1+\alpha Z B_0 \\
    0
    \end{bmatrix}.
    \end{split}
    \end{equation}
Since $(A_{11},B_1)$ is controllable, we use Lemma~\ref{lem:cont_direct} with $M=A_{11}$, \linebreak $M_0=R_1$, $N=B_1$, and $N_0=Z B_0$ to conclude that \mbox{$(A_{11}+\alpha R_1,B_1+\alpha Z B_0)\in\Sigma_\text{cont}$}
for all but at most $n_1^2$ values of $\alpha$. Moreover, since $A_{22}$ is Schur, we have \mbox{$(T\mathcal{A}(\alpha)T^{-1},T\mathcal{B}(\alpha))\in\Sigma_\text{stab}$},
and thus, $(\mathcal{A}(\alpha),\mathcal{B}(\alpha))\in\Sigma_\text{stab}$ for all but at most $n_1^2$ values of $\alpha$. 

Next, we use inclusion \eqref{eq:claim1} to show that statement (a) holds. Let $\mathcal{F}$ be the set of all values of $\alpha$ such that $(\mathcal{A}(\alpha),\mathcal{B}(\alpha))\notin\Sigma_\text{stab}$. By the previous discussion, $\mathcal{F}$ is finite. Since $K$ is a stabilizing gain for all systems within $\Sigma_\mathcal{D}\cap\Sigma_\text{stab}$, we have that
\begin{equation}
\label{eq:beforeSS-1}
\mathcal{A}(\alpha)+\mathcal{B}(\alpha)K=A+BK+\alpha T^{-1}\begin{bmatrix}
Z\\0
\end{bmatrix}(A_0+B_0K)
\end{equation}
is Schur for all $\alpha\in\mathbb{R}\backslash\mathcal{F}$. 
Let $N\in\mathbb{R}^{n\times n_1}$ and \mbox{$M\in\mathbb{R}^{n\times n_2}$} be such that $\begin{bmatrix}
N & M
\end{bmatrix}=(A_0+B_0K)T^{-1}$. Also, define $F\in\mathbb{R}^{m\times n_1}$ and $G\in\mathbb{R}^{m\times n_2}$ by $\begin{bmatrix}
F & G
\end{bmatrix}=KT^{-1}$. We observe that
\begin{equation}
\label{eq:afterSS-1}
T(\mathcal{A}(\alpha)\!+\!\mathcal{B}(\alpha)K)T^{-1}\!\!=\!\!\begin{bmatrix}
    \!A_{11}\!+\!B_1F\!+\!\alpha Z N \!\!&\!\! \!A_{12}\!+\!B_1G\!+\!\alpha ZM \\
    0            \!\!&\!\! A_{22}
    \end{bmatrix}\!.
\end{equation}
Hence, $A_{11}+B_1F+\alpha Z N$ is Schur for all $\alpha\in\mathbb{R}\backslash\mathcal{F}$. It follows now from Lemma~\ref{lem:A+delta*B} that $A_{11}+B_1F$ is Schur and $ZN$ is nilpotent. Note that this argument holds for all \mbox{$Z\in\mathbb{R}^{n_1\times n}$} and all $(A_0,B_0)\in\Sigma_\mathcal{D}^0$. Take $Z=N^\top$. Since $N^\top N$ is symmetric and nilpotent, we have $N=0$. Hence, for any $(A_0,B_0)\in\Sigma_\mathcal{D}^0$ we have \mbox{$A_0+B_0K=\begin{bmatrix}
0 & M
\end{bmatrix}T$}. This implies that
\begin{equation}
\begin{split}
&(A_0+B_0K)\im\begin{bmatrix}
B & AB & \cdots & A^{n-1}B
\end{bmatrix}\\
&=\begin{bmatrix}
0 & M
\end{bmatrix}T\im\begin{bmatrix}
B & AB & \cdots & A^{n-1}B
\end{bmatrix}\\
&=\begin{bmatrix}
0 & M
\end{bmatrix}\im\begin{bmatrix}
B_1 & A_{11}B_1 & \cdots & A^{n-1}_{11}B_1 \\
0   & 0         & \cdots & 0
\end{bmatrix}=\{0\}.
\end{split}
\end{equation}
Therefore, statement (a) holds.

(b) To prove this part, first, we claim that 
\begin{equation}
\label{eq:R(A,B)}
\mathcal{R}(A,B)\subseteq\im X_-\text{ for all } (A,B)\in\Sigma_\mathcal{D}\cap\Sigma_\text{stab}.
\end{equation}
To show this, let $A_0 \in \mathbb{R}^{n \times n}$ be such that \mbox{$\ker A_0 = \im X_-$}. Observe that $(A_0,B_0)\in\Sigma^0_\mathcal{D}$ with \mbox{$B_0=0$}. It follows from part~(a) that \mbox{$A_0\mathcal{R}(A,B)=\{0\}$}. This implies that \mbox{$\mathcal{R}(A,B)\subseteq \ker A_0=\im X_-$}. Now, let $\xi\in\mathbb{R}^{n}$ and $\eta\in\mathbb{R}^{m}$ be such that \mbox{$\begin{bmatrix}
\xi^\top & \eta^\top
\end{bmatrix}^\top\in\ker\begin{bmatrix}
X_-^\top & U_-^\top
\end{bmatrix}$}. Since $(A_\text{true},B_\text{true})\in\Sigma_\mathcal{D}$, for every $Y\in\mathbb{R}^{n}$  we have \mbox{$(A_\text{true}+Y\xi^\top,B_\text{true}+Y\eta^\top)\in\Sigma_\mathcal{D}$}. Since $\Sigma_\text{stab}$ is an open set, there exists a sufficiently small $\varepsilon>0$ such that for every $Y\in\mathbb{R}^{n}$ satisfying $\|Y\|\leq\varepsilon$ we have \mbox{$(A_\text{true}+Y\xi^\top,B_\text{true}+Y\eta^\top)\in\Sigma_\mathcal{D}\cap\Sigma_\text{stab}$}. It follows from \eqref{eq:R(A,B)} that $\im (B_\text{true}+Y\eta^\top)\subseteq\im X_-$ for all $Y\in\mathbb{R}^{n}$ satisfying $\|Y\|\leq\varepsilon$. We use Lemma~\ref{lem:imker_all} with $\hat{N}=B_\text{true}$ and $N_0=\eta^\top$ to conclude that $\eta=0$. Therefore, we have
\begin{equation}
\label{eq:kerX_U_}
\ker\begin{bmatrix}
X_- \\ U_-
\end{bmatrix}^\top=\ker X_-^\top \times \{0\},
\end{equation}
which implies (b). 

(c) To prove this part, it suffices to show that $x(T)\in\im X_-$. If $X_-$ has full row rank, then this condition is obviously satisfied. Suppose that $\rank X_-=r<n$. It follows from part~(b) that \mbox{$T\geq r+m$}. Let $(A,B)\in\Sigma_\mathcal{D}\cap\Sigma_\text{stab}$. Observe that we have
\begin{equation}
\label{eq:state_formula}
x(t)=A^tx(0)+\sum_{k=0}^{t-1} A^kBu(t-k-1)
\end{equation}
for all $t\in[1,T]$. It follows from \eqref{eq:R(A,B)} that the last term satisfies
\begin{equation}
\sum_{k=0}^{t-1} A^kBu(t-k-1)\in\mathcal{R}(A,B)\subseteq \im X_- \text{ for all } t\in[1,T].
\end{equation}
Since $x(t)\in\im X_-$ for all $t\in[1,T-1]$, we have $A^t x(0)\in \im X_-$ for all $t\in[0,T-1]$. As $T-1\geq r$, we use Lemma~\ref{lem:imX_-imX_+} with $\mathcal{V}=\im X_-$ and $M=A$ to conclude that $A^T x(0)\in \im X_-$. Now, it follows from \eqref{eq:state_formula} with $t=T$ that $x(T)\in\im X_-$. 

(d) If $X_-$ has full row rank, then (d) obviously holds. Now, suppose that $X_-$ does not have full row rank. Let $(A,B)\in\Sigma_\mathcal{D}$. First, we show that $\im B\subseteq\im X_-$. For this, observe that since both pairs $(A_\text{true},B_\text{true})$ and $(A,B)$ belong to $\Sigma_\mathcal{D}$, we have
\begin{equation}
\label{eq:AB-ABtrue}
(A-A_\text{true})X_- + (B-B_\text{true})U_-=0.
\end{equation}
Based on \eqref{eq:kerX_U_} and \eqref{eq:AB-ABtrue}, we have $B=B_\text{true}$. It follows now from~\eqref{eq:R(A,B)} that $\im B=\im B_\text{true}\subseteq\mathcal{R}(A_\text{true},B_\text{true})\subseteq\im X_-$. To show that $\im X_-$ is $A$--invariant, we observe from part (c) that \mbox{$A\im X_-+B\im U_-=\im X_+\subseteq\im X_-$}.
Since $\im B\subseteq \im X_-$, we have $A\im X_-\subseteq\im X_-$, which completes the proof.  \hfill \QED

\subsection{Necessary and sufficient conditions with full rank state data}
\label{sec:IV-A}

In case $X_-$ has full row rank, the following theorem shows that data-driven stabilization using stabilizability as prior knowledge is equivalent to data-driven stabilization without prior knowledge. 

\begin{theorem}
\label{th:12}
Suppose that $(A_\text{true},\!B_\text{true})\!\in\!\Sigma_\text{stab}$ and $\rank X_-\!=\!n$. Then, the following statements are equivalent: 
\begin{enumerate}[label=(\alph*),ref=\ref{th:12}(\alph*)]
    \item The data $\mathcal{D}$ are $\Sigma_\text{stab}$--informative for stabilization.
    \item The data $\mathcal{D}$ are informative for stabilization.
\end{enumerate}
Moreover, if (a) holds, then $\Sigma_\mathcal{D}\subseteq\Sigma_\text{stab}$.
\end{theorem}
\begin{proof} 
It is evident that (b) implies (a). To prove that (a) implies (b), first, we assume that $(A_\text{true},B_\text{true})$ is controllable. Since the data $\mathcal{D}$ are $\Sigma_\text{stab}$--informative for stabilization, they are also \mbox{$\Sigma_\text{cont}$--informative} for stabilization. Hence, it follows from Theorem~\ref{th:1} that the data are informative for stabilization. 

Next, we assume that $(A_\text{true},B_\text{true})$ is uncontrollable. Let $K$ be a stabilizing gain for all the systems in $\Sigma_\mathcal{D} \cap \Sigma_\text{stab}$. We show that there exists a $\hat{K}$ that stabilizes all the systems in $\Sigma_\mathcal{D}$. For this, let $\mathcal{S}$ be a subspace satisfying \mbox{$\mathcal{S}\oplus \mathcal{R}(A_\text{true},B_\text{true}) =\mathbb{R}^n $}, where $\oplus$ denotes direct sum. Thus, every $v\in\mathbb{R}^n$ can be written uniquely as \mbox{$v=v_1+v_2$} with $v_1\in\mathcal{R}(A_\text{true},B_\text{true})$ and $v_2\in\mathcal{S}$. We define $\hat{K}$ as the matrix satisfying $\hat{K}v=Kv_1+U_-X_-^\dagger v_2$ for all $v\in\mathbb{R}^n$. Let $(A,B)\in\Sigma_\mathcal{D}$. Take $(A_0,B_0)\in\Sigma_\mathcal{D}^0$ such that $A=A_\text{true}+A_0$ and $B=B_\text{true}+B_0$.  We show that $A_0 + B_0\hat{K}=0$. Since $X_-$ has full row rank, we have $A_0 = -B_0 U_- X_-^\dagger$. Hence, we have
\begin{equation}
(A_0 + B_0\hat{K}) v = A_0 v + B_0 K v_1 + B_0 U_- X_-^\dagger v_2 = (A_0 + B_0 K) v_1.
\end{equation}
It follows from Theorem~\ref{th:nec_cond(b)} that $(A_0 + B_0 K) v_1=0$. Thus, we have $(A_0 + B_0\hat{K}) v=0$ for all $v\in\mathbb{R}^n$, which implies that \mbox{$A_0 + B_0\hat{K}=0$}. Therefore, we have
\begin{equation}
\label{eq:nontrue=true}
A+B\hat{K}=A_\text{true}+A_0+(B_\text{true}+B_0)\hat{K}=A_\text{true}+B_\text{true}\hat{K}.
\end{equation}
Hence, what remains to be proven is that $A_\text{true}+B_\text{true}\hat{K}$ is Schur. Let $\hat{T}\in\mathbb{R}^{n\times n}$ be such that its first $n_1$ columns span $\mathcal{R}(A_\text{true},B_\text{true})$ and the rest of its columns span $\mathcal{S}$. As such, we have
\begin{equation}
\label{eq:Kalman_decomp_T}
\hat{T}^{-1}A_\text{true}\hat{T}=\begin{bmatrix}
\hat{A}_{11} & \hat{A}_{12} \\
0      & \hat{A}_{22}
\end{bmatrix} \hspace{0.25 cm} \text{and} \hspace{0.25 cm} \hat{T}^{-1}B_\text{true}=\begin{bmatrix}
\hat{B}_1 \\ 0
\end{bmatrix},
\end{equation}
where $(\hat{A}_{11},\hat{B}_1)\in\mathbb{R}^{n_1\times n_1}\times \mathbb{R}^{n_1\times m}$ is controllable, \mbox{$\hat{A}_{22}\in\mathbb{R}^{n_2\times n_2}$} is Schur, and $n_1+n_2=n$.  Let $F\in\mathbb{R}^{m\times n_1}$ and $G\in\mathbb{R}^{m\times n_2}$ be defined by $\begin{bmatrix}
F & G
\end{bmatrix}=K\hat{T}$. Since $A_\text{true}+B_\text{true}K$ is Schur, we have that $\hat{A}_{11}+\hat{B}_1F$ is Schur. Now, we observe that
\begin{equation}
\hat{T}^{-1}(A_\text{true}+B_\text{true}\hat{K})\hat{T}=\begin{bmatrix}
\hat{A}_{11}+\hat{B}_1F & \hat{A}_{12}+\hat{B}_1\hat{G} \\
0 & \hat{A}_{22}
\end{bmatrix},
\end{equation}
where $\hat{G}\in\mathbb{R}^{m\times n_2}$ satisfies $\begin{bmatrix}
F & \hat{G}
\end{bmatrix}=\hat{K}\hat{T}$. Since $\hat{A}_{11}+\hat{B}_1F$ and $\hat{A}_{22}$ are both Schur, we have that \mbox{$A_\text{true}+B_\text{true}\hat{K}$} is Schur. Therefore, in view of \eqref{eq:nontrue=true}, $A+B\hat{K}$ is Schur. This implies that the \mbox{data $\mathcal{D}$} are informative for stabilization, hence, (b) holds. For the last part, suppose that (a) holds. Since (a) implies (b), there exists a $K$ such that $A+BK$ is Schur for all $(A,B)\in\Sigma_\mathcal{D}$. Therefore, every \mbox{$(A,B)\in\Sigma_\mathcal{D}$} is stabilizable.
\end{proof}

Due to Theorem \ref{th:12}, in case the state data $X_-$ is of full row rank, stabilizability as prior knowledge does not help in relaxing the conditions for data-driven stabilization. Hence, in this case, a stabilizing feedback gain may still be computed using Proposition~\ref{prop:1}. Nevertheless, if the true system is not controllable, then collecting full rank state data might not be possible. In fact, this depends on the initial condition of the system. For instance, if an uncontrollable system is initially at rest, $x(0)=0$, then $X_-$ will not have full row rank, no matter what input signal is applied to the system. This motivates studying the case where the state data is rank-deficient, which is the topic of the next section.

\subsection{Necessary and sufficient conditions with rank-deficient state data}
\label{sec:IV-B}

In case $X_-$ does not have full row rank, the data are not informative for stabilization, and Proposition~\ref{prop:1} fails to provide a stabilizing feedback. Interestingly, it turns out that in this case, one may be able to find a stabilizing feedback from data by incorporating the prior knowledge on stabilizability. 

\begin{theorem}
\label{th:3}
Suppose that $(A_\text{true},\!B_\text{true})\!\in\!\Sigma_\text{stab}$ and $\rank X_-\!<\!n$. Then, the data $\mathcal{D}$ are $\Sigma_\text{stab}$--informative for stabilization \emph{if and only if} the following conditions hold:
\begin{enumerate}[label=(\alph*),ref=\ref{th:3}(\alph*)]
    \item\label{th:3(a)} $\im X_+\subseteq\im X_-$,
    \item\label{th:3(b)} $\im \begin{bmatrix}
    X_- \\ U_-
    \end{bmatrix}=\im X_- \times \mathbb{R}^m$.
\end{enumerate}
\end{theorem}\vspace{0.25cm}

Theorem~\ref{th:3} provides a full characterization of $\Sigma_\text{stab}$--informativity for stabilization in case $\rank X_-<n$. Before providing a proof, we first also consider the problem of computing a stabilizing feedback gain from $\Sigma_\text{stab}$--informative data. For this, given data $\mathcal{D}$, let \mbox{$r\coloneqq \rank X_-$}, $S\in\mathbb{R}^{n\times n}$ be nonsingular, and \mbox{$\hat{X}_- \in \mathbb{R}^{r \times T}$} be of full row rank such that
\begin{equation}
\label{eq:def:SXhat_-}
SX_-=\begin{bmatrix}
\hat{X}_- \\ 0
\end{bmatrix}.
\end{equation}
These matrices $S$ and $\hat{X}_-$ can be computed, for example, using a QR decomposition of $X_-$. Moreover, let $\hat{X}_+\in\mathbb{R}^{r\times n}$ be defined as
\begin{equation}
\label{eq:def:SXhat_+}
\hat{X}_+\coloneqq\begin{bmatrix}
I_r & 0
\end{bmatrix}SX_+.
\end{equation}

\begin{proposition}
\label{prop:K=[K1 K2]}
Suppose that $(A_\text{true},B_\text{true})\in\Sigma_\text{stab}$, the data $\mathcal{D}$ are $\Sigma_\text{stab}$--informative for stabilization, and \mbox{$\rank X_-<n$}. Then, the following statements hold:
\begin{enumerate}[label=(\alph*),ref=\ref{prop:K=[K1 K2]}(\alph*)]
    \item\label{prop:K=[K1 K2](a)} There exists $\Theta\in\mathbb{R}^{T\times r}$ such that the following LMI is feasible:
\begin{equation}
\label{eq:dd_lmi_stab}
\hat{X}_-\Theta=\Theta^\top \hat{X}_-^\top\ \text{ and }\ \begin{bmatrix}
\hat{X}_-\Theta & \hat{X}_+\Theta \\
\Theta^\top \hat{X}_+^\top & \hat{X}_-\Theta
\end{bmatrix}>0.
\end{equation}
\item\label{prop:K=[K1 K2](b)} Suppose that $\Theta$ satisfies LMI \eqref{eq:dd_lmi_stab}. Let $K=\begin{bmatrix}
K_1 & K_2
\end{bmatrix}S$, where $K_1=U_-\Theta(\hat{X}_-\Theta)^{-1}$ and $K_2\in\mathbb{R}^{m\times (n-r)}$ is arbitrary. Then, $A+BK$ is Schur for all  $(A,B)\in\Sigma_\mathcal{D}\cap \Sigma_\text{stab}$. 
\end{enumerate}
\end{proposition}

To prove Theorem~\ref{th:3} and Proposition~\ref{prop:K=[K1 K2]}, we need an auxiliary result presented in the following lemma. 
This lemma shows that if the data satisfy conditions (a) and (b) of Theorem~\ref{th:3}, then $S$ can be used as a state transformation to simultaneously factorize all members of $\Sigma_\mathcal{D}\cap\Sigma_\text{stab}$ into a stabilizable part and an autonomous, stable part. 

\begin{lemma}
\label{lem:partition_proof}
Suppose that the data $\mathcal{D}$ satisfy $\rank X_-<n$. Let $(A,B)\in\Sigma_\mathcal{D}\cap\Sigma_\text{stab}$. If statements (a) and (b) of Theorem~\ref{th:3} hold, then we have
\begin{equation}
\label{eq:partition_proof}
SAS^{-1}=\begin{bmatrix}
A_{11} & A_{12} \\
0 & A_{22}
\end{bmatrix}\ \text{ and }\ SB=\begin{bmatrix}
B_1 \\ 0
\end{bmatrix},
\end{equation}
where the pair $(A_{11},B_1)\in\mathbb{R}^{r\times r}\times\mathbb{R}^{r\times m}$ is stabilizable and \mbox{$A_{22}\in\mathbb{R}^{(n-r)\times (n-r)}$} is Schur. Moreover, we have
\begin{equation}
\label{eq:partition_formula_A_11_B1}
\begin{bmatrix}
A_{11} & B_1
\end{bmatrix}=\hat{X}_+\begin{bmatrix}
\hat{X}_- \\
U_-
\end{bmatrix}^\dagger. 
\end{equation}
\end{lemma}\vspace{0.25 cm}
\begin{proof}
Write $SAS^{-1}=\begin{bmatrix}
A_{11} & A_{12} \\
A_{21} & A_{22}
\end{bmatrix}$ and \mbox{$SB=\begin{bmatrix}
B_1 \\ B_2
\end{bmatrix}$},
where $A_{21} \in \mathbb{R}^{(n-r)\times r}$ and $B_2 \in \mathbb{R}^{(n-r) \times m}$. First, we show that $A_{21}$ and $B_2$ are both equal to zero. Since $\im X_+\subseteq \im X_-$, we have \mbox{$SX_+=\begin{bmatrix}
\hat{X}_+^\top & 0
\end{bmatrix}^\top$}. It follows from \mbox{$SX_+=(SAS^{-1})SX_-+SBU_-$} that \mbox{$0=A_{21}\hat{X}_-+B_{2}U_-$} and $\hat{X}_+=A_{11}\hat{X}_-+B_{1}U_-$. Since \mbox{$\im \begin{bmatrix}
X_-^\top &
U_-^\top
\end{bmatrix}^\top=\im X_- \times \mathbb{R}^m$}, we have $\rank\begin{bmatrix}
\hat{X}_-^\top &
U_-^\top
\end{bmatrix}=r+m$. Thus, \mbox{$0=A_{21}\hat{X}_-+B_{2}U_-$} implies that $A_{21}=0$ and $B_2=0$. Now, we observe that $(A,B)\in\Sigma_\text{stab}$ implies $(SAS^{-1},SB)\in\Sigma_\text{stab}$. This, together with $B_2=0$ and $A_{21}=0$, implies that $A_{22}$ is Schur and $(A_{11},B_1)$ is stabilizable. Finally, the formula in \eqref{eq:partition_formula_A_11_B1} follows immediately from $\hat{X}_+=\begin{bmatrix}
A_{11} & B_{1}
\end{bmatrix}\begin{bmatrix}
\hat{X}_-^\top & U_-^\top
\end{bmatrix}^\top$ and the fact that $\begin{bmatrix}
\hat{X}_-^\top &
U_-^\top
\end{bmatrix}$ is of full column rank. 
\end{proof}

We note that, unlike Kalman decomposition, the data-driven decomposition provided by Lemma~\ref{lem:partition_proof} does not guarantee the pair $(A_{11},B_1)$ to be controllable. Now, we use Lemma~\ref{lem:imker_all} and Theorem~\ref{th:nec_cond} to prove Theorem~\ref{th:3}. 

\textit{Proof of Theorem~\ref{th:3}:} The ``only if'' part follows immediately from parts (b) and (c) of Theorem~\ref{th:nec_cond}. To prove the ``if'' part, assume that conditions (a) and (b) hold. Since $(A_\text{true},B_\text{true})\in\Sigma_\mathcal{D}\cap\Sigma_\text{stab}$, it follows from Lemma~\ref{lem:partition_proof} that
\begin{equation}
\label{eq:partition_proof_true}
SA_\text{true}S^{-1}=\begin{bmatrix}
\hat{A}_{11} & \hat{A}_{12} \\
0 & \hat{A}_{22}
\end{bmatrix}\ \text{ and }\ SB_\text{true}=\begin{bmatrix}
\hat{B}_1 \\ 0
\end{bmatrix},
\end{equation}
where $(\hat{A}_{11},\hat{B}_1)\in\mathbb{R}^{r\times r}\times \mathbb{R}^{r\times m}$ is stabilizable and \mbox{$\hat{A}_{22}\in\mathbb{R}^{(n-r)\times (n-r)}$} is Schur. Since $(A_\text{true},B_\text{true})$ is stabilizable, there exists a $K$ such that $A_\text{true}+B_\text{true}K$ is Schur. We show that $K$ is a stabilizing feedback gain for all systems belonging to $\Sigma_\mathcal{D}\cap\Sigma_\text{stab}$. For this, let $K_1\in\mathbb{R}^{m\times r}$ and $K_2\in\mathbb{R}^{m\times (n-r)}$ be defined by $\begin{bmatrix}
K_1 & K_2
\end{bmatrix}=KS^{-1}$ and observe that
\begin{equation}
S(A_\text{true}+B_\text{true}K)S^{-1}=\begin{bmatrix}
\hat{A}_{11}+\hat{B}_1K_1 & \hat{A}_{12}+\hat{B}_1K_2 \\
0 & \hat{A}_{22}
\end{bmatrix}
\end{equation}
is Schur. This implies that $\hat{A}_{11}+\hat{B}_1K_1$ is Schur. Now, let \mbox{$(A,B)\in\Sigma_\mathcal{D}\cap\Sigma_\text{stab}$}. Define $A_{11}$, $A_{12}$, $A_{22}$, and $B_1$ as in \eqref{eq:partition_proof}. In view of Lemma~\ref{lem:partition_proof}, $A_{22}$ is Schur and we have
\begin{equation}
\label{eq:id_cont_part}
\begin{bmatrix}
A_{11} & B_1
\end{bmatrix}=\begin{bmatrix}
\hat{A}_{11} & \hat{B}_1
\end{bmatrix}=\hat{X}_+\begin{bmatrix}
\hat{X}_- \\
U_-
\end{bmatrix}^\dagger. 
\end{equation}
Therefore, since $\hat{A}_{11}+\hat{B}_1K_1$ is Schur, we have that $A_{11}+B_1K_1$ is Schur. Hence,
\begin{equation}
S(A+BK)S^{-1}=\begin{bmatrix}
A_{11}+B_1K_1 & A_{12}+B_1K_2 \\
0 & A_{22}
\end{bmatrix},
\end{equation}
is Schur. Therefore, $A+BK$ is Schur, which implies that the \mbox{data $\mathcal{D}$} are $\Sigma_\text{stab}$--informative for stabilization.   \hfill \QED

\textit{Proof of Proposition~\ref{prop:K=[K1 K2]}.} (a) Since the data $\mathcal{D}$ are \mbox{$\Sigma_\text{stab}$--informative} for stabilization, the conditions (a) and (b) in Theorem~\ref{th:3} hold. Let $\hat{A}_{11}$, $\hat{A}_{12}$, $\hat{A}_{22}$, and $\hat{B}_1$ be defined as in \eqref{eq:partition_proof_true}. It follows from Lemma~\ref{lem:partition_proof} that $\hat{A}_{22}$ is Schur. Let $K$ be such that \mbox{$A_\text{true}+B_\text{true}K$} is Schur. Define $K_1\in\mathbb{R}^{m\times n_1}$ and $K_2\in\mathbb{R}^{m\times n_2}$ by \mbox{$\begin{bmatrix}
K_1 & K_2
\end{bmatrix}=KS^{-1}$}. Observe that
\begin{equation}
S(A_\text{true}+B_\text{true}K)S^{-1}=\begin{bmatrix}
\hat{A}_{11}+\hat{B}_1K_1 & \hat{A}_{12}+\hat{B}_1K_2 \\
0 & \hat{A}_{22}
\end{bmatrix}.
\end{equation}
Since $S(A_\text{true}+B_\text{true}K)S^{-1}$ is Schur, this implies that $\hat{A}_{11}+\hat{B}_1K_1$ is Schur. Now, it follows from \eqref{eq:id_cont_part} that
\begin{equation}
\label{eq:A_{11}+B_1K_1}
\hat{A}_{11}+\hat{B}_1K_1=\hat{X}_+\begin{bmatrix}
\hat{X}_- \\
U_-
\end{bmatrix}^\dagger\begin{bmatrix}
I \\ K_1
\end{bmatrix}.
\end{equation}
Let $P>0$ be such that the Lyapunov inequality \mbox{$P-(\hat{A}_{11}+\hat{B}_1K_1)P(\hat{A}_{11}+\hat{B}_1K_1)^\top>0$} holds. Now, take $\Theta=\begin{bmatrix}
\hat{X}_- \\
U_-
\end{bmatrix}^\dagger\begin{bmatrix}
I \\ K_1
\end{bmatrix}P$.
We show that, with this $\Theta$, the LMI \eqref{eq:dd_lmi_stab} is feasible. To this end, we show that the lower block of \eqref{eq:dd_lmi_stab}, $\hat{X}_-\Theta$, and its Schur complement with respect to the lower block, $\hat{X}_-\Theta-\hat{X}_+\Theta(\hat{X}_-\Theta)^{-1}\Theta^\top \hat{X}_+^\top$, are positive definite. First, we note that $\hat{X}_-\Theta=P>0$ and $\hat{X}_+\Theta P^{-1}=\hat{A}_{11}+\hat{B}_1K_1$. Hence, we have
\begin{equation}
\begin{split}
\hat{X}_-\Theta-\hat{X}_+\Theta(\hat{X}_-\Theta)^{-1}\Theta^\top \hat{X}_+^\top=
P-\hat{X}_+\Theta P^{-1}\Theta^\top \hat{X}_+^\top \\
=P-(\hat{A}_{11}+\hat{B}_1K_1)P(\hat{A}_{11}+\hat{B}_1K_1)^\top>0.
\end{split}
\end{equation}
Therefore, LMI \eqref{eq:dd_lmi_stab} is feasible. 

(b) Assume that $\Theta$ satisfies LMI \eqref{eq:dd_lmi_stab}. Let \mbox{$(A,B)\in\Sigma_\mathcal{D}\cap\Sigma_\text{stab}$}, $K_1=U_-\Theta(\hat{X}_-\Theta)^{-1}$, $K_2\in\mathbb{R}^{m\times r}$, and $K=\begin{bmatrix}
K_1 & K_2
\end{bmatrix}S$. We show that $A+BK$ is Schur. We note that conditions (a) and (b) of Theorem~\ref{th:3} hold. Hence, let $A_{11}$, $A_{12}$, $A_{22}$, and $B_1$ be defined as in Lemma~\ref{lem:partition_proof}. Hence,
\begin{equation}
S(A+BK)S^{-1}=\begin{bmatrix}
A_{11}+B_1K_1 & A_{12}+B_1K_2 \\
0 & A_{22}
\end{bmatrix}.
\end{equation}
Since $A_{22}$ is Schur, it suffices now to show that \mbox{$A_{11}+B_1K_1$} is Schur. Since $\hat{X}_+=A_{11}+B_1U_-$, we have
\begin{equation}
\begin{split}
&A_{11}+B_1K_1=A_{11}+B_1U_-\Theta(\hat{X}_-\Theta)^{-1}\\
&=(A_{11}\hat{X}_-+B_1U_-)\Theta(\hat{X}_-\Theta)^{-1}=\hat{X}_+\Theta(\hat{X}_-\Theta)^{-1}.
\end{split}
\end{equation} 
Now, based on a Schur complement argument, we observe that LMI~\eqref{eq:dd_lmi_stab} implies
\begin{equation}
\begin{split}
&\hat{X}_-\Theta-\hat{X}_+\Theta(\hat{X}_-\Theta)^{-1}\Theta^\top \hat{X}_+^\top \\
&=\hat{X}_-\Theta-(A_{11}+B_1K_1)(\hat{X}_-\Theta)(A_{11}+B_1K_1)^\top>0.
\end{split}
\end{equation}
Since $\hat{X}_-\Theta>0$, this implies that $A_{11}+B_1K_1$ is Schur. \hfill \QED

We close this section with three remarks.

\begin{remark}
\label{rem:18}
Suppose that the data $\mathcal{D}$ are \mbox{$\Sigma_\text{stab}$--informative} for stabilization, but they are \emph{not} informative for stabilization. In this case, the following facts hold: Theorem~\ref{th:1} implies that $\Sigma_\mathcal{D}\cap\Sigma_\text{cont}$ is empty, hence, the true system is stabilizable but not controllable; Theorems~\ref{th:12} and~\ref{th:3} imply that the state data satisfy $\rank X_-<n$; Theorem~\ref{th:3(b)} implies that $B_\text{true}$ can be uniquely recovered from the data, i.e., every $(A,B)\in\Sigma_\mathcal{D}$ is such that $B=B_\text{true}$; Theorems~\ref{th:12} and~\ref{th:3} imply that there exists $(A,B)\in\Sigma_\mathcal{D}$ that is not stabilizable, i.e., the data $\mathcal{D}$ are not informative for \emph{stabilizability} because $X_+ - X_-$ does not have full row rank, see \cite[Def. 7]{van2020data} and \cite[Thm. 8]{van2020data}.
\end{remark}

\begin{remark}
It was shown in \cite[Sec. IV]{van2021matrix} that if there exists a $K$ such that $A+BK$ is Schur for all $(A,B)\in\Sigma_\mathcal{D}$, then there exists a $P>0$ such that the Lyapunov inequality \mbox{$P-(A+BK) P(A+BK)^\top>0$} holds for all $(A,B)\in\Sigma_\mathcal{D}$, i.e., the data $\mathcal{D}$ are informative for stabilization \emph{if and only if} they are informative for quadratic stabilization (see \cite[Def. 2]{van2021matrix}). However, we note that this does not necessarily hold in case the data are \mbox{$\Sigma_\text{stab}$--informative} for stabilization. This can be shown using Example~\ref{ex:1}. In that example, the closed-loop system for every member of $\Sigma_\mathcal{D}\cap\Sigma_\text{stab}$ is of the form $A+BK=\begin{bmatrix}
0 & \alpha \\
0 & \beta
\end{bmatrix}$ for some $\alpha\in\mathbb{R}$ and $|\beta|<1$. All such matrices are Schur. However, we show that there does not exist a common $P$ for all these matrices. Aiming for a contradiction, suppose that there exists a positive definite \mbox{$P\coloneqq\begin{bmatrix}
P_{11} & P_{12} \\
P_{12} & P_{22}
\end{bmatrix}\in\mathbb{R}^{2\times 2}$} with the property that the Lyapunov inequality,
\begin{equation}
\label{eq:ex1_Lyap}
\begin{bmatrix}
P_{11} & P_{12} \\
P_{12} & P_{22}
\end{bmatrix}-\begin{bmatrix}
0 & \alpha \\
0 & \beta
\end{bmatrix}\begin{bmatrix}
P_{11} & P_{12} \\
P_{12} & P_{22}
\end{bmatrix}\begin{bmatrix}
0 & 0 \\
\alpha & \beta
\end{bmatrix}
>0,
\end{equation}
holds for all $\alpha\in\mathbb{R}$ and $|\beta|<1$. Multiply \eqref{eq:ex1_Lyap} from left and right, respectively, by $\begin{bmatrix}
1 & 0
\end{bmatrix}$ and $\begin{bmatrix}
1 & 0
\end{bmatrix}^\top$ to obtain \mbox{$P_{11}-\alpha^2 P_{22}>0$}. Since $P_{22}>0$, it is evident that there exists a sufficiently large $\alpha$ such that this inequality does not hold. Hence, we reach a contradiction. Therefore, the Lyapunov inequality \eqref{eq:ex1_Lyap} does not hold for all closed-loop systems consistent with the data and the prior knowledge.
\end{remark}

\begin{remark}
Let $\mathbb{C}_{g}^\text{dis}\coloneqq\set{\lambda\in\mathbb{C}}{|\lambda|<1}$ be the discrete-time stability domain. Conditions (a) and (b) of Theorem~\ref{th:3} imply that there exists a $K$ such that $\spec(A+BK)\subset\mathbb{C}_{g}^\text{dis}$ for all \mbox{$(A,B)\in\Sigma_\mathcal{D}\cap\Sigma_\text{stab}$}. Nevertheless, these two conditions do not explicitly depend on $\mathbb{C}_{g}^\text{dis}$. In fact, one can verify that Lemma~\ref{lem:partition_proof}, and consequently, the ``if'' part of Theorem~\ref{th:3}, can be generalized to an arbitrary stability domain $\mathbb{C}_g$, see \cite[Def. 2.12]{trentelmancontrol}. In that case, the prior knowledge on stabilizability is replaced by stabilizability with respect to the specific choice of $\mathbb{C}_g$. This has applications, e.g., in continuous-time systems, where the stability domain is \mbox{$\mathbb{C}_{g}^\text{con}\coloneqq\set{\lambda\in\mathbb{C}}{\mathfrak{R}(\lambda)<0}$}. We leave a detailed study of arbitrary stability domains as future work.
\end{remark}

\section{Numerical Example}
To demonstrate our results, we consider a three-tank system, depicted in Fig. \ref{fig:three_tank}, which is described by the continuous-time model,
\begin{equation}
\dot{x}_1= -k_1 x_1 + u, \ \dot{x}_2=k_{2}x_3-u, \text{ and } \dot{x}_3=-k_3x_3,
\end{equation}
where the flow coefficients are $k_{1}=0.1$ and \mbox{$k_{2}=k_3=0.5$}. The third mode of the system, $x_3$, is uncontrollable. This lack of controllability arises from the directional flow from tank 3 to 2.
The system is discretized using the zero-order hold method with a sampling time of $0.1$, yielding the matrices
\begin{equation}
A_{\text{true}} =
\begin{bmatrix}
0.99  &  0 &  0 \\
0 & 1 & 0.0488 \\
0 & 0 & 0.9512
\end{bmatrix}\ \text{ and }\
B_{\text{true}} =
\begin{bmatrix}
0.0995 \\
-0.1000 \\
0
\end{bmatrix}.
\end{equation}
We consider the prior knowledge on the true system as $\Sigma_\text{pk}=\Sigma_\text{stab}$. We perform an open-loop experiment of length $T=5$, during which we collect the following input-state data:\vspace{-0.25 cm}
\begin{table}[h]
    \centering
    \begin{tabular}{c|cccccc}
       $t$ & $0$ & $1$ & $2$ & $3$ & $4$ & $5$ \\ \hline
       $u(t)$ &  $1$ & $1$ & $1$ & $0$ & $1$ &  \\ \hline
        \multirow{3}{*}{$x(t)$} & $0$ & $0.0995$ & $0.198$ & $0.2955$ & $0.2926$ & $0.3892$ \\
        & $0$ &  $-0.1$ & $-0.2$ & $-0.3$ & $-0.3$ & $-0.4$ \\
        & $0$ & $0$ & $0$ & $0$ & $0$ & $0$ 
    \end{tabular}\vspace{-0.25 cm}
\end{table}

\begin{figure}
    \centering
    \begin{tikzpicture}[scale=1,cross/.style={path picture={ 
    \draw[black](path picture bounding box.south east) -- (path picture bounding box.north west) (path picture bounding box.south west) -- (path picture bounding box.north east);}}]
    \draw (0,0.3)--(0,-0.25)--(1.4,-0.25)--(1.4,-0.7);
    \draw (1.6,-0.7)--(1.6,-0.25)--(3,-0.25)--(3,0.3);
    \node [] at (1.5,0){\footnotesize Tank 3};
    \draw[|<->|] (-0.25,0.25)--(-0.25,-0.25);
    \node[] at ( -0.6, 0 ) {$x_3$};
    \draw (0,-0.5)--(0,-1.25)--(1.6,-1.25);
    \draw (1.6,-1.25)--(3,-1.25)--(3,-1)--(3.25,-1);
    \draw (3.25,-0.8)--(3,-0.8)--(3,-0.5);
    \node [] at (1.5,-0.95){\footnotesize Tank 2};
    \draw[|<->|] (-0.25,-0.65)--(-0.25,-1.25);
    \node[] at ( -0.6, -0.95 ) {$x_2$};
    \draw (3.5,-0.5)--(3.5,-0.8)--(3.25,-0.8);
    \draw (3.25,-1)--(3.5,-1)--(3.5,-1.25)--(4.9,-1.25)--(4.9,-1.75);
    \draw (5.1,-1.75)--(5.1,-1.25)--(6.5,-1.25)--(6.5,-0.5);
    \node [] at (5,-0.95){\footnotesize Tank 1};
    \draw[|<->|] (6.75,-1.25)--(6.75,-0.65);
    \node[] at ( 7.1, -0.95 ) {$x_1$};
    \node [draw,circle,minimum width=0.01 cm,fill=red!20] at (3.25,-0.9){};
    \node [] at (3.25,-1.3){$u$};
    \draw[fill=black] (3.13,-0.9) -- (3.22,-0.95) -- (3.22,-0.85) -- cycle;
    \draw[fill=black] (3.37,-0.9) -- (3.28,-0.85) -- (3.28,-0.95) -- cycle;
    \draw (3.5,-1.65)--(3.5,-2.1)--(6.5,-2.1)--(6.5,-1.65);
    \node [] at (4.95,-1.9){\footnotesize Basin};
    \begin{pgfonlayer}{background}
        \filldraw[blue!10] (0,0.25)--(0,-0.25)--(1.4,-0.25)--(1.4,-0.65)--(1.6,-0.65)--(1.6,-0.25)--(3,-0.25)--(3,0.25)--cycle;
        \filldraw[blue!10] (0,-0.65)--(0,-1.25)--(1.6,-1.25)--(3,-1.25)--(3,-1)--(3.25,-1)--(3.25,-0.8)--(3,-0.8)--(3,-0.65)--cycle;
        \filldraw[blue!10] (3.5,-0.65)--(3.5,-0.8)--(3.25,-0.8)--(3.25,-1)--(3.5,-1)--(3.5,-1.25)--(4.9,-1.25)--(4.9,-1.75)--(5.1,-1.75)--(5.1,-1.25)--(6.5,-1.25)--(6.5,-0.65)--cycle;
        \filldraw[blue!10] (3.5,-1.7)--(3.5,-2.1)--(6.5,-2.1)--(6.5,-1.7)--cycle;
    \end{pgfonlayer}
    \end{tikzpicture}
    \caption{Schematic representation of the three-tank system.}\vspace{-0.5 cm}
    \label{fig:three_tank}
\end{figure}
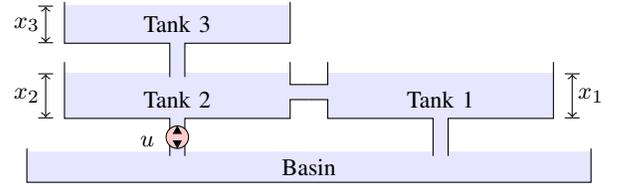

Here, $X_{-}$ has rank $r=2$. A matrix $S$, satisfying \eqref{eq:def:SXhat_-}, can be simply taken to be $S = I$.
We solve \eqref{eq:dd_lmi_stab} using MATLAB\footnote{The MATLAB code for this example is available at https://github.com/TrenBaltussen/Data-Driven-Stabilization.} with YALMIP \cite{yalmip} and MOSEK \cite{mosek}, which yields
\begin{equation}
\Theta \!=\! 10^4\begin{bmatrix}
-2.2514 & 4.456 & 2.2921 & 1.5349 & -4.5532 \\
-0.9315 & 1.7108 & 1.9234 & 0.7591 & -2.6451
\end{bmatrix}^\top\!\!.
\end{equation}
Following Proposition \ref{prop:K=[K1 K2]}, we compute $ K_1 = \begin{bmatrix} -3.5289 & 6.251 \end{bmatrix}$ and set $K_2$ to zero, which yields $K = \begin{bmatrix}
    -3.5289 & 6.251 & 0
\end{bmatrix}$ that is a stabilizing feedback gain for all systems within $\Sigma_\mathcal{D}\cap\Sigma_{\text{stab}}$.

To illustrate the advantage of incorporating prior knowledge of stabilizability, we perform Monte Carlo simulations with $1000$ scenarios. For each scenario, we simulate the system from $t=0$ to $t = 100$ with the input sequence and the entries of the initial condition drawn independently at random from a Poisson distribution with parameter $\lambda = 1$. To investigate the effect of the number of samples on the informativity of the data, we use the first $T$ samples for each round of analysis. Here, we take $T=3$, $4$, $5$, $10$, and finally, the entire dataset $T=100$, see Table \ref{tab:MC_Analysis}. We see that for $T = 3$, none of the scenarios yield informative data for system identification because $T<n+m$, and thus, $\small \begin{bmatrix} X_-^\top & U_-^\top \end{bmatrix}$ does not have full column rank, which is a necessary and sufficient condition for system identification \cite[Prop. 6]{van2020data}. Nevertheless, in this case, $29.3\%$ of the datasets are informative for stabilization (without using prior knowledge). At this point, $65.82 \%$ of the datasets are $\Sigma_{\text{stab}}$--informative for stabilization. By increasing the number of samples, the percentage of datasets that are informative for identification approaches that of stabilization (without using prior knowledge). Both of these numbers eventually reach $63.93\%$ and remain unchanged. Interestingly, with $T\geq10$, we see that all the datasets are $\Sigma_{\text{stab}}$--informative for stabilization.
This clearly demonstrates the advantage of incorporating stabilizability as prior knowledge in data-driven control.

\begin{table}[t]
\centering
\caption{Informativity of randomly generated data.}
\label{tab:MC_Analysis}
\begin{tabular}{cccc}
\toprule
& \multirow{2}{*}{\parbox{3cm}{\centering \vspace{1mm} Informative for\\identification \cite[Def. 5]{van2020data}}} &  \multicolumn{2}{c}{$\Sigma_\text{pk}$--informative for stabilization} \\
\cmidrule(l){3-4}
$T$  &  & $\Sigma_\text{pk}=\mathcal{M}$ & $\Sigma_\text{pk}=\Sigma_\text{stab}$ \\
\midrule
$3$ & $0 \% $ & $29.30 \% $ & $65.82 \% $ \\
$4$ & $61.93 \% $ & $62.74 \% $ & $99.10 \% $ \\
$5$ & $63.23 \% $ & $63.55 \% $ & $99.58 \% $ \\
$10$ & $63.93 \% $ & $63.93 \% $ & $100 \% $ \\
$100$ & $63.93 \% $ & $63.93 \% $ & $100 \% $ \\
\bottomrule
\end{tabular}\vspace{-0.3 cm}
\end{table}

\section{Conclusions}

In this work, data-driven stabilization using prior knowledge on controllability and stabilizability has been studied. It has been shown that data-driven stabilization using controllability as prior knowledge is equivalent to data-driven stabilization without prior knowledge. It has also been shown that if the state data satisfy a rank condition, then incorporating stabilizability as prior knowledge does not lead to weaker conditions on the data. For the case where the state data are rank deficient, however, it has been shown that data-driven stabilization incorporating stabilizability as prior knowledge requires weaker conditions on the data when compared to data-driven stabilization without prior knowledge. A somewhat curious outcome of the paper is that knowledge of stabilizability may weaken the conditions on the data, while controllability, a stronger property, does not. This is due to the fact that in the former scenario, there exists the possibility that \emph{none} of the data-consistent systems are controllable.

A class of prior knowledge that has not been studied in this work is the one representing an upper bound on the dimension of the reachable subspace of the system. In that case, one expects that obtaining a data-driven feedback gain requires an even weaker condition. In addition, here, we only focused on noise-free data. Data-driven stabilization using such prior knowledge in the presence of noisy data is also an interesting topic that is left as future work. 

\section*{References}

\bibliographystyle{IEEEtran}
\bibliography{biblo}

\end{document}